\documentclass[12pt]{article}

\usepackage{amsmath, epsfig, cite, setspace}
\usepackage{amssymb}
\usepackage{amsfonts}
\usepackage{latexsym}
\usepackage{amsthm}

\makeatletter
\renewcommand{\@seccntformat}[1]{{\csname the#1\endcsname}.\hspace{.5em}}
\makeatother

\newtheorem{thm}{Theorem}[section]

\newtheorem{lem}[thm]{Lemma}

\renewcommand{\qed}{\hfill$\Box$\medskip}
\renewcommand{\thefootnote}{*}

\numberwithin{equation}{section}

\begin{document}

\begin{center}
{\large\bf On two congruences involving Ap\'{e}ry and Franel numbers}
\end{center}

\vskip 2mm \centerline{Guo-Shuai Mao}
\begin{center}
{\footnotesize $^1$Department of Mathematics, Nanjing
University of Information Science and Technology, Nanjing 210044,  People's Republic of China\\
{\tt maogsmath@163.com  } }
\end{center}


\vskip 0.7cm \noindent{\bf Abstract.}
In this paper, we generalise a congruence which proved by V. J.-W. Guo and J. Zeng \cite{gz-jnt-2012} involving Ap\'{e}ry numbers,
and we obtain a congruence involving Franel numbers which confirms a congruence conjecture of Z.-W. Sun \cite[Conjecture 57(ii)]{sun-njdx-2019}.



\renewcommand{\thefootnote}{**}

\section{Introduction}
Recall that the Bernoulli polynomials are defined as
$$B_n(x)=\sum_{k=0}^n\binom nkB_kx^{n-k}\ \ (n\in\mathbb{N}).$$

The $n$th generalised harmonic number is given by 
$$H_n^{(k)}=\sum_{j=1}^n\frac1{j^k}$$
and the $n$th classic harmonic number is known as $H_n=H_n^{(1)}$.

The Ap\'{e}ry numbers are defined by
$$A_n=\sum_{k=0}^n\binom{n}{k}^2\binom{n+k}k^2\ \ (n\in\mathbb{N}).$$
In the past decades, congruences involving Ap\'{e}ry numbers have attracted the attention of many researchers (see, for instance, \cite{beu-jnt-1987,ccc-jnt-1980,gessel-jnt-1982,gz-jnt-2012}).
In 2012, Guo and Zeng proved the following congruence conjectured by Sun \cite{sun-jnt-2012}: If $p>3$ is a prime, then
\begin{align*}
&\frac1p\sum_{k=0}^{p-1}(-1)^k(2k+1)A_k\equiv \left(\frac{p}{3}\right)\pmod{p^2}.
\end{align*}
where $\left(\frac{\cdot}{p}\right)$ denotes the Legendre symbol.

The main objective of this paper is to prove the following result, which generalises the above congruence.
\begin{thm}\label{ThAp} Let $p>3$ be a prime. Then we have
\begin{align}
&\frac1p\sum_{k=0}^{p-1}(-1)^k(2k+1)A_k\equiv \left(\frac{p}{3}\right)+\frac{p^2}6B_{p-2}\left(\frac13\right)\pmod{p^3}.\label{1pak}
\end{align}
\end{thm}
The well known Franel numbers are defined as
$$f_n=\sum_{k=0}^n\binom{n}k^3\ \ (n\in\mathbb{N}),$$
which satisfy the recurrence:
$$(n+1)^2f_n=(7n^2+7n+2)f_n+8n^2f_{n-1}.$$
Strehl \cite{str-mtn-1993} gave the following identity:
\begin{align}
f_n=\sum_{k=0}^n\binom{n}k\binom{k}{n-k}\binom{2k}k=\sum_{k=0}^n\binom{n}k^2\binom{2k}n.\label{fnidentity}
\end{align}
Many people (see \cite{g-itsf-2013,jv-rama-2010,sun-aam-2013,sun-jnt-2013}) have studied congruences for Franel numbers. For instance, Sun [Theorem 1.1]\cite{sun-aam-2013} proved that for any prime $p>3$,
$$\sum_{k=0}^{p-1}(-1)^kf_k\equiv\left(\frac{p}{3}\right)\pmod{p^2}.$$
The second aim of this paper is to prove the following congruence which conjectured by Sun \cite{sun-njdx-2019}.
\begin{thm}\label{Thfp} For any prime $p>3$, we have
\begin{align}
&\sum_{k=0}^{p-1}(-1)^kf_k\equiv \left(\frac{p}{3}\right)+\frac{2p^2}3B_{p-2}\left(\frac13\right)\pmod{p^3}.\label{1pfk}
\end{align}
\end{thm}
We end the introduction by giving the organization of this paper.
We shall give some preliminary results first in Section 2, and then prove Theorem \ref{ThAp} in Section 3, Theorem \ref{Thfp} will be proved in Section 4 in two ways.
\section{Preliminary lemmas}
\begin{lem}\label{3-3} Let $p>3$ be a prime. Then
$$\sum_{k=1}^{(p-1)/2}\frac{(-3)^k}k\sum_{j=1}^k\frac1{(2j-1)(-3)^j}\equiv\frac16B_{p-2}\left(\frac13\right)\pmod{p}.$$
\end{lem}
\begin{proof}
First we have
\begin{align*}
&\sum_{k=1}^{(p-1)/2}\frac{(-3)^k}k\sum_{j=1}^k\frac1{(2j-1)(-3)^j}=\sum_{k=1}^{(p-1)/2}\frac1k\sum_{j=1}^k\frac{(-3)^{k-j}}{2j-1}=\sum_{k=1}^{(p-1)/2}\frac1k\sum_{j=0}^{k-1}\frac{(-3)^j}{2k-2j-1}\\
&=\sum_{j=0}^{(p-3)/2}(-3)^j\sum_{k=j+1}^{(p-1)/2}\frac1{k(2k-2j-1)}=\sum_{j=0}^{(p-3)/2}\frac{(-3)^j}{2j+1}\sum_{k=j+1}^{(p-1)/2}\left(\frac2{2k-2j-1}-\frac1k\right).
\end{align*}
It is easy to check that
$$\sum_{k=j+1}^{(p-1)/2}\frac2{2k-2j-1}=\sum_{k=1}^{(p-1)/2-j}\frac2{2k-1}=2\left(H_{p-1-2j}-\frac12H_{\frac{p-1}2-j}\right)=2H_{p-1-2j}-H_{\frac{p-1}2-j}.$$
So
$$\sum_{k=1}^{(p-1)/2}\frac{(-3)^k}k\sum_{j=1}^k\frac1{(2j-1)(-3)^j}=\sum_{j=0}^{(p-3)/2}\frac{(-3)^j}{2j+1}\left(2H_{p-1-2j}-H_{\frac{p-1}2-j}-H_{\frac{p-1}2}+H_j\right).$$
It is easy to see that $H_{p-1-2j}\equiv H_{2j}\pmod p$ and $H_{\frac{p-1}2-j}-H_{\frac{p-1}2}\equiv 2H_{2j}-H_j\pmod p$.
Hence 
$$\sum_{k=1}^{(p-1)/2}\frac{(-3)^k}k\sum_{j=1}^k\frac1{(2j-1)(-3)^j}\equiv\sum_{j=0}^{(p-3)/2}\frac{(-3)^j}{2j+1}\left(2H_j-2H_{\frac{p-1}2}\right)\pmod p.$$
It is obvious that
\begin{align*}
\sum_{j=0}^{(p-3)/2}\frac{(-3)^j}{2j+1}\left(H_j-H_{\frac{p-1}2}\right)&=\sum_{j=1}^{(p-1)/2}\frac{(-3)^{(p-1)/2-j}}{p-2j}\left(H_{\frac{p-1}2-j}-H_{\frac{p-1}2}\right)\\
&\equiv\frac{(-3)^{(p-1)/2}}2\sum_{j=1}^{(p-1)/2}\frac{H_j-2H_{2j}}{j(-3)^j}\pmod p.
\end{align*}
Therefore in view of \cite[(1.3)]{liu-arxiv-2020}, and by the fact that $(-3)^{(p-1)/2}\equiv\left(\frac{-3}p\right)=\left(\frac{p}3\right)\pmod p$, we immediately get the desired result.
\end{proof}
\begin{lem}\label{2j2k} For any prime $p>3$, we have
$$\sum_{k=1}^{(p-1)/2}\frac{1}{k\binom{2k}k}\sum_{j=1}^k\frac{\binom{2j}j}j\equiv\frac13B_{p-2}\left(\frac13\right)\pmod{p}.$$
\end{lem}
\begin{proof}
In view of \cite{sun-scm-2011}, we have
$$\sum_{j=1}^{(p-1)/2}\frac{\binom{2j}j}j\equiv0\pmod p.$$
This, with $H_{(p-1)/2}^{(2)}\equiv0\pmod p$ yields that
\begin{align*}
\sum_{k=1}^{(p-1)/2}\frac{1}{k\binom{2k}k}\sum_{j=1}^k\frac{\binom{2j}j}j&\equiv-\sum_{k=1}^{(p-1)/2}\frac{1}{k\binom{2k}k}\sum_{j=k+1}^{(p-1)/2}\frac{\binom{2j}j}j=-\sum_{j=2}^{(p-1)/2}\frac{\binom{2j}j}j\sum_{k=1}^{j-1}\frac1{k\binom{2k}k}\\
&=-\sum_{j=1}^{(p-1)/2}\frac{\binom{2j}j}j\sum_{k=1}^{j}\frac1{k\binom{2k}k}+H_{(p-1)/2}^{(2)}\\
&\equiv-\sum_{j=1}^{(p-1)/2}\frac{\binom{2j}j}j\sum_{k=1}^{j}\frac1{k\binom{2k}k}\pmod p.
\end{align*}
It is well known that $\binom{2j}j\equiv\binom{(p-1)/2}j(-4)^j\pmod p$ for each $j=0,1,\ldots,(p-1)/2$, so
$$\sum_{k=1}^{(p-1)/2}\frac{1}{k\binom{2k}k}\sum_{j=1}^k\frac{\binom{2j}j}j\equiv-\sum_{j=1}^{(p-1)/2}\frac{\binom{(p-1)/2}j(-4)^j}j\sum_{k=1}^{j}\frac1{k\binom{2k}k}\pmod p.$$
By \textit{Sigma} we have the following identity
$$\sum_{j=1}^{n}\frac{\binom{n}j(-4)^j}j\sum_{k=1}^{j}\frac1{k\binom{2k}k}=-2\sum_{k=1}^{n}\frac{(-3)^k}k\sum_{j=1}^k\frac1{(2j-1)(-3)^j}.$$
Setting $n=(p-1)/2$ in the above identity, we have
$$\sum_{k=1}^{(p-1)/2}\frac{1}{k\binom{2k}k}\sum_{j=1}^k\frac{\binom{2j}j}j\equiv2\sum_{k=1}^{(p-1)/2}\frac{(-3)^k}k\sum_{j=1}^k\frac1{(2j-1)(-3)^j}\pmod p,$$
then we immediately get the desired result by Lemma \ref{3-3}.
\end{proof}
\begin{lem}\label{p-2k} Let $p>3$ be a prime. Then for each $k=1,2,\ldots,(p-1)/2$ we have
$$\sum_{j=2}^{p-2k}\frac{\binom{k+j-1}{k}}{k+j}\equiv(-1)^k\left(\frac32\sum_{j=1}^k\frac{\binom{2j}j}j-\frac{\binom{2k}k}k\right)-\frac1{k+1}\pmod p.$$
\end{lem}
\begin{proof}
First we have
$$\sum_{j=2}^{p-2k}\frac{\binom{k+j-1}{k}}{k+j}=\sum_{j=k+2}^{p-k}\frac{\binom{j-1}k}j=\sum_{j=k+1}^{p-k}\frac{\binom{j-1}k}j-\frac1{k+1}.$$
In view of \cite[(4.37)]{g-online}, we have
$$\sum_{j=a}^n(-1)^{j-a}\binom{n}j\frac1j=\sum_{j=a}^n\binom{j-1}{a-1}\frac1j\ \ (1\leq a\leq n).$$
Hence 
\begin{align*}
\sum_{j=k+1}^{p-k}\frac{\binom{j-1}k}j&=\sum_{j=k+1}^{p-k}(-1)^{j-k-1}\binom{p-k}j\frac1j\\
&=\sum_{j=1}^{p-k}(-1)^{j-k-1}\binom{p-k}j\frac1j-\sum_{j=1}^{k}(-1)^{j-k-1}\binom{p-k}j\frac1j.
\end{align*}
By \textit{Sigma} we have $\sum_{j=1}^n(-1)^j\binom{n}j\frac1j=-H_n$, so
$$\sum_{j=k+1}^{p-k}\frac{\binom{j-1}k}j=(-1)^kH_{p-k}-\sum_{j=1}^{k}(-1)^{j-k-1}\binom{p-k}j\frac1j.$$
It is known that $\binom{-n}k=(-1)^k\binom{n+k-1}k$, hence
$$\sum_{j=1}^{k}(-1)^{j-k-1}\binom{p-k}j\frac1j\equiv\sum_{j=1}^{k}(-1)^{j-k-1}\binom{-k}j\frac1j=(-1)^{k+1}\sum_{j=1}^k\frac1j\binom{k+j-1}j\pmod p.$$
By \textit{Sigma} we have
$$\sum_{j=1}^n\frac1j\binom{n+j-1}j=\frac{1+3n}n-\frac{\binom{2n}n}n-H_n+\frac32\sum_{j=2}^n\frac{\binom{2j}j}j.$$
Thus 
\begin{align*}
\sum_{j=1}^{k}(-1)^{j-k-1}\binom{p-k}j\frac1j&\equiv(-1)^{k+1}\left(\frac{1+3k}k-\frac{\binom{2k}k}k-H_k+\frac32\sum_{j=2}^k\frac{\binom{2j}j}j\right)\\
&=(-1)^{k+1}\left(\frac32\sum_{j=1}^k\frac{\binom{2j}j}j-\frac{\binom{2k}k}k-H_{k-1}\right)\pmod p.
\end{align*}
Since $H_{p-k}\equiv H_{k-1}\pmod p$, therefore
$$\sum_{j=2}^{p-2k}\frac{\binom{k+j-1}{k}}{k+j}\equiv(-1)^{k}\left(\frac32\sum_{j=1}^k\frac{\binom{2j}j}j-\frac{\binom{2k}k}k\right)-\frac1{k+1}\pmod p.$$
Now the proof of Lemma \ref{p-2k} is complete.
\end{proof}
\begin{lem}\label{p-2k-k} Let $p>3$ be a prime. Then for each $k=1,2,\ldots,(p-1)/2$ we have 
$$\sum_{j=0}^{p-2k}\binom{-k}{k+j}\frac{(-1)^{k+j}}{\binom{k+j}k}\equiv\frac{3k}2\sum_{j=1}^k\frac{\binom{2j}j}j-\frac32\binom{2k}k\pmod p.$$
\end{lem}
\begin{proof}
We know that 
\begin{align*}
\sum_{j=0}^{p-2k}\binom{-k}{k+j}\frac{(-1)^{j}}{\binom{k+j}k}&\equiv\sum_{j=0}^{p-2k}\binom{p-k}{k+j}\frac{1}{\binom{p-k-1}j}=\sum_{j=0}^{p-2k}\frac{\binom{p-k}{p-2k-j}}{\binom{p-k-1}j}\\
&=\sum_{j=0}^{p-2k}\binom{p-k}{j}\frac{1}{\binom{p-k-1}{p-2k-j}}\pmod p.
\end{align*}
By \textit{Sigma}, we have the following identity,
$$\sum_{j=1}^n\frac{\binom{n+k}j}{\binom{n+k-1}{n-j}}=-\frac{n+k}{k+1}-\frac{n}{k\binom{n+k-1}{n-1}}+\frac{(n+k)^2\binom{n+k-1}{n-1}}{nk}-(n+k)\sum_{j=2}^n\frac{\binom{k+j-1}k}{k+j}.$$
Setting $n=p-2k$, we have
$$\sum_{j=1}^{p-2k}\frac{\binom{p-k}j}{\binom{p-k-1}{p-2k-j}}\equiv\frac{k}{k+1}+\frac{2(-1)^k}{\binom{2k}{k}}-\frac{(-1)^k}{2}\binom{2k}k+k\sum_{j=2}^{p-2k}\frac{\binom{k+j-1}k}{k+j}\pmod p.$$
Hence
\begin{align*}
\sum_{j=0}^{p-2k}\binom{-k}{k+j}\frac{(-1)^{j}}{\binom{k+j}k}&\equiv\frac{k}{k+1}+\frac{2(-1)^k}{\binom{2k}{k}}-\frac{(-1)^k}{2}\binom{2k}k+k\sum_{j=2}^{p-2k}\frac{\binom{k+j-1}k}{k+j}+\frac1{\binom{-k-1}{k-1}}\\
&=\frac{k}{k+1}-\frac{(-1)^k}{2}\binom{2k}k+k\sum_{j=2}^{p-2k}\frac{\binom{k+j-1}k}{k+j}\pmod p.
\end{align*}
By Lemma \ref{p-2k}, we have
\begin{align*}
\sum_{j=0}^{p-2k}\binom{-k}{k+j}\frac{(-1)^{k+j}}{\binom{k+j}k}&\equiv(-1)^k\frac{k}{k+1}-\frac{1}{2}\binom{2k}k+(-1)^kk\sum_{j=2}^{p-2k}\frac{\binom{k+j-1}k}{k+j}\\
&\equiv(-1)^k\frac{k}{k+1}-\frac{1}{2}\binom{2k}k+k\left(\frac32\sum_{j=1}^k\frac{\binom{2j}j}j-\frac{\binom{2k}k}k\right)-\frac{(-1)^k}{k+1}\\
&=\frac{3k}2\sum_{j=1}^k\frac{\binom{2j}j}j-\frac32\binom{2k}k\pmod p.
\end{align*}
Now we finish the proof of Lemma \ref{p-2k-k}.
\end{proof}
\begin{lem}\label{2kkp+1} For any prime $p>3$, we have
$$\sum_{k=(p+1)/2}^{p-1}\binom{2k}k\sum_{j=p-k}^k\binom{k}j\binom{k+j}j(-1)^{k+j}\equiv-p^2B_{p-2}\left(\frac13\right)\pmod{p^3}.$$
\end{lem}
\begin{proof}
It is known that $\binom{k+j}j\equiv0\pmod p$ for each $p-k\leq j\leq p-1$ with $k\in\{(p+1)/2,\ldots,p-1\}$, and in view of \cite{sun-scm-2011} 
$$k\binom{2k}k\binom{2(p-k)}{p-k}\equiv-2p\pmod{p^2}\ \mbox{for}\ \mbox{all}\ k=1,2,\ldots,(p-1)/2.$$
So 
\begin{align*}
&\sum_{k=(p+1)/2}^{p-1}\binom{2k}k\sum_{j=p-k}^k\binom{k}j\binom{k+j}j(-1)^{k+j}\\
&=\sum_{k=1}^{(p-1)/2}\binom{2p-2k}{p-k}\sum_{j=k}^{p-k}\binom{p-k}j\binom{p-k+j}j(-1)^{p-k+j}\\
&\equiv-2p\sum_{k=1}^{(p-1)/2}\frac1{k\binom{2k}k}\sum_{j=k}^{p-k}\binom{p-k}j\binom{p-k+j}j(-1)^{k+j+1}\pmod{p^3}.
\end{align*}
It is obvious that
$$\binom{p-k+j}j=\frac{(p-k+j)\cdots(p-k+1)}{j!}\equiv\frac{p(j-k)!(-1)^{k-1}(k-1)!}{j!}=\frac{p(-1)^{k-1}}{k\binom{j}k}\pmod{p^2}.$$
Hence
\begin{align*}
&\sum_{k=(p+1)/2}^{p-1}\binom{2k}k\sum_{j=p-k}^k\binom{k}j\binom{k+j}j(-1)^{k+j}\\
&\equiv-2p^2\sum_{k=1}^{(p-1)/2}\frac1{k^2\binom{2k}k}\sum_{j=k}^{p-k}\binom{-k}j\frac{(-1)^j}{\binom{j}k}\\
&=-2p^2\sum_{k=1}^{(p-1)/2}\frac1{k^2\binom{2k}k}\sum_{j=0}^{p-2k}\binom{-k}{k+j}\frac{(-1)^{k+j}}{\binom{k+j}k}\pmod{p^3}.
\end{align*}
By Lemma \ref{p-2k-k}, Lemma \ref{2j2k} and $H_{(p-1)/2}^{(2)}\equiv0\pmod p$, we have
\begin{align*}
\sum_{k=(p+1)/2}^{p-1}\binom{2k}k\sum_{j=p-k}^k\binom{k}j\binom{k+j}j(-1)^{k+j}&\equiv-2p^2\sum_{k=1}^{(p-1)/2}\frac1{k^2\binom{2k}k}\left(\frac{3k}2\sum_{j=1}^k\frac{\binom{2j}j}j-\frac32\binom{2k}k\right)\\
&=-3p^2\sum_{k=1}^{(p-1)/2}\frac1{k^2\binom{2k}k}\sum_{j=1}^k\frac{\binom{2j}j}j+3p^2H_{(p-1)/2}^{(2)}\\
&\equiv-p^2B_{p-2}\left(\frac13\right)\pmod{p^3}.
\end{align*}
Therefore the proof of Lemma \ref{2kkp+1} is finished.
\end{proof}
\section{Proof of Theorem \ref{ThAp}}
{\it Proof of Theorem \ref{ThAp}.} By \cite[Thm 2.1]{gz-jnt-2012}, we have
\begin{align*}
\frac1p\sum_{k=0}^{p-1}(-1)^k(2k+1)A_k=\sum_{k=0}^{p-1}\binom{2k}k\sum_{j=0}^k\binom{k}j\binom{k+j}j\binom{p-1}{k+j}\binom{p+k+j}{k+j}.
\end{align*}
It is easy to see that 
$$\binom{p-1}{k+j}\binom{p+k+j}{k+j}=\prod_{i=1}^{k+j}\frac{p^2-i^2}{i^2}\equiv(-1)^{k+j}\left(1-p^2H_{k+j}^{(2)}\right)\pmod{p^3}.$$
So 
\begin{align*}
\frac1p\sum_{k=0}^{p-1}(-1)^k(2k+1)A_k\equiv \theta_1+\theta_2\pmod{p^3},
\end{align*}
where 
$$\theta_1=\sum_{k=0}^{(p-1)/2}\binom{2k}k\sum_{j=0}^k\binom{k}j\binom{k+j}j(-1)^{k+j}\left(1-p^2H_{k+j}^{(2)}\right),$$
$$\theta_2=\sum_{k=(p+1)/2}^{p-1}\binom{2k}k\sum_{j=0}^{p-1-k}\binom{k}j\binom{k+j}j(-1)^{k+j}.$$
It is obvious that
\begin{align*}
\theta_2+\sum_{k=0}^{(p-1)/2}\binom{2k}k\sum_{j=0}^k\binom{k}j\binom{k+j}j(-1)^{k+j}&=\sum_{k=0}^{p-1}\binom{2k}k\sum_{j=0}^k\binom{k}j\binom{k+j}j(-1)^{k+j}\\
&-\sum_{k=(p+1)/2}^{p-1}\binom{2k}k\sum_{j=p-k}^{k}\binom{k}j\binom{k+j}j(-1)^{k+j}.
\end{align*}
By Chu-Vandermonde identity, we have 
$$\sum_{k=0}^{p-1}\binom{2k}k\sum_{j=0}^k\binom{k}j\binom{k+j}j(-1)^{k+j}=\sum_{k=0}^{p-1}\binom{2k}k.$$
In view of Lemma \ref{2kkp+1}, we have
$$\theta_2+\sum_{k=0}^{(p-1)/2}\binom{2k}k\sum_{j=0}^k\binom{k}j\binom{k+j}j(-1)^{k+j}\equiv\sum_{k=0}^{p-1}\binom{2k}k+p^2B_{p-2}\left(\frac13\right)\pmod{p^3}.$$
Hence
\begin{align*}
&\frac1p\sum_{k=0}^{p-1}(-1)^k(2k+1)A_k-\sum_{k=0}^{p-1}\binom{2k}k-p^2B_{p-2}\left(\frac13\right)\\
&\equiv-p^2\sum_{k=0}^{(p-1)/2}\binom{2k}k\sum_{j=0}^k\binom{k}j\binom{k+j}j(-1)^{k+j}H_{k+j}^{(2)}\\
&=-p^2\sum_{k=0}^{(p-1)/2}\binom{2k}k\sum_{j=k}^{2k}\binom{k}{j-k}\binom{j}{k}(-1)^{j}H_{j}^{(2)}\pmod{p^3}.
\end{align*}
By \cite[(4.1)]{liu-arxiv-2020}, we have
$$\sum_{j=k}^{2k}\binom{k}{j-k}\binom{j}{k}(-1)^{j}H_{j}^{(2)}=3\sum_{j=1}^k\frac1{j^2\binom{2j}j}.$$
So
$$\frac1p\sum_{k=0}^{p-1}(-1)^k(2k+1)A_k-\sum_{k=0}^{p-1}\binom{2k}k-p^2B_{p-2}\left(\frac13\right)\equiv-3p^2\sum_{k=1}^{(p-1)/2}\binom{2k}k\sum_{j=1}^k\frac1{j^2\binom{2j}j}\pmod{p^3}.$$
Therefore, by \cite{st-jnt-2013} and \cite[(4.4)]{liu-arxiv-2020}, we immediately get the desired result
$$\frac1p\sum_{k=0}^{p-1}(-1)^k(2k+1)A_k\equiv\left(\frac{p}3\right)+\frac{p^2}6B_{p-2}\left(\frac13\right)\pmod{p^3}.$$
Now we finish the proof of Theorem \ref{ThAp}.\qed
\section{Proof of Theorem \ref{Thfp}}
{\it Proof of Theorem \ref{Thfp}.} First we have the following identity
$$\sum_{k=0}^n\binom{n}k\binom{n+1+k}kf_k=\frac{(-1)^n}{n+1}\sum_{k=0}^n(-1)^k(2k+1)A_k.$$
Setting $n=p-1$ in the above identity, we have 
$$\sum_{k=0}^{p-1}\binom{p-1}k\binom{p+k}kf_k=\frac{1}{p}\sum_{k=0}^{p-1}(-1)^k(2k+1)A_k.$$
It is easy to check that
$$\binom{p-1}{k}\binom{p+k}{k}=\prod_{j=1}^{k}\frac{p^2-j^2}{j^2}\equiv(-1)^{k}\left(1-p^2H_{k}^{(2)}\right)\pmod{p^3}.$$
So we have
$$\sum_{k=0}^{p-1}(-1)^k\left(1-p^2H_k^{(2)}\right)f_k\equiv\frac{1}{p}\sum_{k=0}^{p-1}(-1)^k(2k+1)A_k\pmod{p^3}.$$
This, with Theorem \ref{ThAp} and \cite[(1.5)]{liu-arxiv-2020} yields that
$$\sum_{k=0}^{p-1}(-1)^kf_k\equiv\left(\frac{p}3\right)+\frac{2p^2}3B_{p-2}\left(\frac13\right)\pmod{p^3}.$$
Now the proof of Theorem \ref{Thfp} is completed.\qed

\noindent{\it Another proof of Theorem \ref{Thfp}.} In view of (\ref{fnidentity}), we have
\begin{align*}
\sum_{k=0}^{p-1}(-1)^kf_k&=\sum_{k=0}^{p-1}(-1)^k\sum_{j=0}^k\binom{k}j\binom{j}{k-j}\binom{2j}j=\sum_{j=0}^{p-1}\binom{2j}j\sum_{k=j}^{p-1}\binom{k}j\binom{j}{k-j}(-1)^k\\
&=\sum_{j=0}^{\frac{p-1}2}\binom{2j}j\sum_{k=j}^{2j}\binom{k}j\binom{j}{k-j}(-1)^k+\sum_{j=\frac{p+1}2}^{p-1}\binom{2j}j\sum_{k=j}^{p-1}\binom{k}j\binom{j}{k-j}(-1)^k\\
&=\sum_{j=0}^{\frac{p-1}2}\binom{2j}j\sum_{k=0}^{j}\binom{k+j}j\binom{j}{k}(-1)^{k+j}+\sum_{j=\frac{p+1}2}^{p-1}\binom{2j}j\sum_{k=0}^{p-1-j}\binom{k+j}j\binom{j}{k}(-1)^{k+j}.
\end{align*}
By Chu-Vandermonde identity, we have
$$\sum_{k=0}^{j}\binom{k+j}j\binom{j}{k}(-1)^{k+j}=1$$
Hence
\begin{align*}
\sum_{k=0}^{p-1}(-1)^kf_k&=\sum_{j=0}^{\frac{p-1}2}\binom{2j}j+\sum_{j=\frac{p+1}2}^{p-1}\binom{2j}j(-1)^j\left(\sum_{k=0}^{j}\binom{k+j}j\binom{j}{k}(-1)^k-\sum_{k=p-j}^{j}\binom{k+j}j\binom{j}{k}(-1)^k\right)\\
&=\sum_{j=0}^{p-1}\binom{2j}j-\sum_{j=\frac{p+1}2}^{p-1}\binom{2j}j\sum_{k=p-j}^{j}\binom{k+j}j\binom{j}{k}(-1)^{k+j}.
\end{align*}
By Lemma \ref{2kkp+1} and \cite{st-jnt-2013}, we have
$$\sum_{k=0}^{p-1}(-1)^kf_k\equiv\left(\frac{p}3\right)+\frac{2p^2}3B_{p-2}\left(\frac13\right)\pmod{p^3}.$$
Therefore the proof of Theorem \ref{Thfp} is complete.\qed

\vskip 3mm \noindent{\bf Acknowledgments.}
The author is funded by the Startup Foundation for Introducing Talent of Nanjing University of Information Science and Technology (2019r062).

\end{document}